\newfont{\Bbb}{msbm10 scaled\magstephalf}
\documentclass[11.0pt,reqno]{amsart}
\usepackage{amsfonts}
\usepackage{amsmath, amsthm, amscd, amsfonts}
\usepackage{amssymb}
\usepackage{amsmath}
\usepackage{xcolor}
\usepackage{amscd}
\usepackage{graphicx}
 \usepackage{epstopdf}
\allowdisplaybreaks[4]
 \newtheorem{thm}{Theorem}[section]
 \newtheorem{cor}[thm]{Corollary}
 \newtheorem{lem}[thm]{Lemma}
 \newtheorem{prop}[thm]{Proposition}
 \theoremstyle{definition}
 \newtheorem{defn}[thm]{Definition}
\theoremstyle{remark}
 \newtheorem{rem}[thm]{Remark}
 \newtheorem{exm}[thm]{Example}
 \numberwithin{equation}{section}
\newcommand{\pf}{\begin{proof}}
\newcommand{\zb}{\end{proof}}

\newcommand{\ma}{\mathcal}

\def\CC{{\mathbb C}}
\def\DD{{\mathbb D}}
\def\dim{\mathop{\rm dim}\nolimits}

\def\Ker{\mathop{\rm ker}\nolimits}
\def\NN{{\mathbb N}}

\begin{document}

\title[cyclic nearly invariant subspaces]{Cyclic nearly  invariant subspaces for semigroups of isometries}
\author[Y. Liang]{Yuxia Liang}
\address{Yuxia Liang \newline School of Mathematical Sciences,
Tianjin Normal University, Tianjin 300387, P.R. China.} \email{liangyx1986@126.com}
\author[J. R. Partington]{Jonathan R. Partington}
\address{Jonathan R. Partington \newline School of Mathematics,
  University of Leeds, Leeds LS2 9JT, United Kingdom.}
 \email{J.R.Partington@leeds.ac.uk}
\begin{abstract} In this paper, the structure of the nearly invariant subspaces for discrete semigroups generated by several (even infinitely many) automorphisms of the unit disc is described.
As part of this work,
the near $S^*$-invariance property of the image space $C_\varphi(\ker T)$ is explored for
composition operators $C_\varphi$, induced by
inner functions $\varphi$, and  Toeplitz operators $T$. After that, the analysis of nearly invariant subspaces for strongly continuous multiplication  semigroups of isometries is developed with a study of cyclic subspaces generated by a single Hardy class function. These are characterised in terms of model spaces in all cases when the outer factor is a product of an invertible function and a rational
(not necessarily invertible) function. Techniques used include the
theory of Toeplitz kernels and reproducing kernels.
\end{abstract}

\maketitle

\section{Introduction}
Study of the structure of invariant subspaces for particular classes of operators has produced a series of significant theorems and examples, successfully building relations with complex function theory, operator theory, and functional analysis. Especially, shift-invariant subspaces in Hardy space $H^2(\mathbb{D})$ of the unit disc are described by the well-known Beurling theorem
\cite[A.~Cor.~1.4.1]{Nik}. In order to classify the shift-invariant subspaces of the Hardy space of an annulus, Hitt introduced the nearly invariant subspaces on the Hardy space of the unit disc in \cite{hitt}, which was also  a generalisation
of Hayashi's results concerning Toeplitz kernels in \cite{Ha}.  Sarason continued to explore nearly invariant subspaces in \cite{Sa1,Sa2} revealing further  relations with  the kernels of Toeplitz operators. This work promoted  the understanding of particular subspaces of Hardy spaces.

Afterwards, the study of nearly invariant subspaces was extended to the vectorial case by Chalendar,  Chevrot and Partington in \cite{CCP10}.   C\^{a}mara and Partington continued with some systematic investigations of near invariance and Toeplitz kernels in \cite{CaP,CaP2}. More recently, many rather striking applications of nearly invariant subspaces have been provided. On the one hand,  Hartmann and Ross combined the truncated Toeplitz operator with nearly invariant subspaces of $H^2(\mathbb{D})$ and characterized the boundary behaviour of functions in \cite{HR}. On the other hand, Aleman,  Baranov,  Belov and  Hedenmalm  characterized the structure of the backward shift-invariant and nearly invariant subspaces in weighted Fock-type spaces of entire functions in \cite{ABBH}, which inspires the extension of investigation of nearly invariant subspace into more analytic spaces, such as  Brangesian spaces in \cite{ALS}. Moreover, O'Loughlin \cite{ryan}
established an application of nearly invariant subspaces to the theory of truncated
Toeplitz operators.

Questions concerning the structure of nearly invariant subspaces have led to developments in this branch of operator theory and may, therefore, be considered   important   in the study of linear bounded operators on  general separable Hilbert spaces. More recently, we posed the definitions of nearly invariant subspaces for a left invertible operator $T$ and  $C_0$ semigroup $\{T(t)\}_{t\geq 0}$ on a separable infinite-dimensional Hilbert space $\ma{H}$, which are the nontrivial generalisations of unilateral shift operators on Hardy space.
Recall that a bounded linear operator $T$ on $\ma{H}$ is a shift operator if $T$ is an isometry and $\|T^{*n} f\|\rightarrow 0$ for all $f\in \mathcal{H}$ as $n\rightarrow \infty$ (see, e.g. \cite[Chapter 1]{RR}).  The authors have characterized the nearly $T^{-1}$-invariant subspace for a pure shift operator $T$ with finite multiplicity  (i.e., the dimension of $\Ker T^*$ is finite) in terms of invariant subspaces on  vector-valued Hardy space under the backward shift (see, \cite[Theorem 2.4]{LP2}).

When we turn to consider the pure shift operator $T$ with infinite multiplicity, the Toeplitz operator $T_\theta$ with  infinite Blaschke product $\theta$ is taken as the preliminary example. Since $T_\theta^*$ is   universal in the sense of Rota (see \cite{Ro}), it can be directly connected with the shift semigroup $\{S(t)\}_{t\geq 0}$ on $L^2(0,\infty)$. Although we employ model spaces in Hardy space to equivalently construct some examples about minimal nearly invariant subspaces in \cite{LP3}, we are still stuck with the general formula of nearly $\{S(t)^*\}_{t\geq 0}$ invariant subspaces. In this sense, the main aim of this paper is to investigate the properties of cyclic nearly invariant subspaces for semigroups of isometries, including the discrete semigroup $\{T_{\psi_1^m\psi_2^n}:\;m,n\in \mathbb{N}_0\}$,  where $\NN_0 = \NN \cup \{0\}$,  generated by two different automorphisms $\psi_1,\;\psi_2$ of the unit disc and also the shift semigroup $\{S(t)\}_{t\geq 0}$ on $L^2(0,\infty)$,
as defined below in  \eqref{shift}. The subspaces in these examples show a new understanding of near invariance, composition operators and model spaces.

For the convenience of the readers, we first recall some relevant definitions on $H^2(\mathbb{D})$. The unilateral shift  $S:\; H^2(\mathbb{D})\rightarrow H^2(\mathbb{D})$  is defined as $[Sf](z)=zf(z).$ Its adjoint operator on $H^2(\mathbb{D})$ is the backward shift  $[S^*f](z)= (f(z)-f(0))/z.$
Given $\phi\in L^\infty(\mathbb{T}),$  the Toeplitz operator $T_\phi: \;H^2(\mathbb{D})\rightarrow H^2(\mathbb{D})$ is defined as $$(T_\phi f)(\lambda)=P_{H^2}(\phi \cdot f)(\lambda)=\int_{\mathbb{T}}\frac{\phi(\zeta)f(\zeta)}
{1-\overline{\zeta}\lambda}dm(\zeta),$$ where $P_{H^2}$ is the orthogonal projection from $L^2(\mathbb{T})$ onto $H^2(\mathbb{D}).$ Here $L^2(\mathbb{T})= H^2(\mathbb{D})\oplus \overline{zH^2(\mathbb{D})}.$ Moreover, the space $H^\infty(\mathbb{D})$ (or $H^\infty(\mathbb{C}_+)$)  is the Banach algebra of bounded analytic functions on $\mathbb{D}$ (or $\mathbb{C}_+$).
Furthermore, we say $f$ belongs to the Smirnov class if $f\in H(\mathbb{D})$ and
 $$\lim\limits_{r\rightarrow 1^{-}}\int_{\mathbb{T}}\log(1+|f(rz)|)dm(z)=
 \int_{\mathbb{T}}\log(1+|f(z)|)dm(z)<\infty.$$

 To be specific, a closed subspace $\ma{M}\subseteq H^2(\mathbb{D})$ is said to be nearly $S^*$-invariant (or weakly invariant) if whenever $f\in \ma{M}$ and $f(0)=0,$ then $S^*f\in\ma{M}.$ Informally, $\ma{M}\subseteq H^2(\mathbb{D})$ is  nearly $S^*$-invariant if the zeros of functions in $\ma{M}$ can be divided out without leaving the space. Hitt  formulated the most widely known characterization of nearly $S^*$-invariant subspaces in $H^2(\mathbb{D})$.

\begin{thm}\cite[Proposition 3]{hitt} \label{thm Hitt}The nearly $S^*$-invariant subspaces of $H^2(\DD)$ have the form $\ma{M}=uK$, with $u\in \ma{M}$ of unit norm, $u(0)>0,$ and $u$ orthogonal to all elements of $\ma{M}$ vanishing at the origin, $K$  an $S^*$-invariant subspace, and the operator of multiplication by $u$  isometric from $K$ into $H^2(\mathbb{D})$.
\end{thm}

Note that the shift operator $T$ is left invertible, so the authors defined the  nearly $T^{-1}$-invariant subspace in $\ma{H}$ as below. Given a left invertible $T\in \ma{B}(\ma{H})$, a subspace $\ma{M}\subseteq \ma{H}$ is said to be nearly $T^{-1}$-invariant if whenever $g\in \ma{H}$ such that $Tg\in \ma{M},$ then $g\in \ma{M}$ (see, \cite[Definition 1.2]{LP2}). Based on our work on  the characterizations of nearly $T^{-1}$-invariant subspace of a shift operator $T$ with finite multiplicity, we obtained a vector-valued characterization of the nearly $T_{B_m}^{-1}$-invariant subspaces in $H^2(\mathbb{D})$ when $B_m$ is a finite Blaschke product with degree $m$ (see, \cite[Corollary 2.6]{LP2}).

 It is well known that  a general Hardy class function can be decomposed as the product
of  an inner factor and an outer factor.  Examples of inner functions include Blaschke products. In particular, when a Blaschke product has degree $1,$ it is also called an  automorphism and denoted by $\psi$ in this paper. It is clear that $\psi$ is a one-to-one analytic map  $\mathbb{D}$ onto $\mathbb{D},$ given by $$\psi(z)=\lambda \frac{a-z}{1-\overline{a}z},
\quad \hbox{where} \quad |\lambda|=1\;\mbox{and}\;a\in \mathbb{D}.$$

 For any finite Blaschke product,  there is a Taylor series converging uniformly on the closed disc (in fact with absolutely summable coefficients). Then truncating the series  and taking the polynomials of  $S$, we see that \textit{all the  Beurling-type invariant subspaces $\theta H^2(\mathbb{D})$ for some inner function $\theta$ are also  $T_\psi$-invariant subspaces, and conversely. Meanwhile, the invariant subspaces  for  $T_{\psi}^{-1}$   are the model spaces $K_\theta=H^2\ominus \theta H^2.$ }

 The above research leads us naturally to ask two direct questions.\vspace{0.5mm}

  (1) {\it{Is there a Hitt-like formula for nearly $T_\psi^{-1}$-invariant subspace for the automorphism $\psi$? }}

(2) {\it{How can we represent the nearly $T_{\theta}^{-1}$-invariant subspaces in $H^2(\mathbb{D})$ {\rm(or $H^2(\mathbb{C_+})$ )} for an infinite degree Blaschke product $\theta$?}}

 For  Question (1), this will be solved by using composition operators $C_\varphi$ and model spaces $K_\theta$ (see Section 2). After that, we continue to examine the nearly invariant subspaces with respect to the finitely-generated semigroup $\{T_{\psi_1^m\psi_2^n}^{-1}:\;m,n\in \mathbb{N}_0\}$ with automorphisms $\psi_1$ and $\psi_2$ on $\mathbb{D}.$ During this process, we obtain  several interesting results. Especially, for any non-automorphic inner function $\varphi,$ we prove that  the subspace $C_\varphi(\Ker T)$ is not  nearly $S^*$-invariant for any Toeplitz operator $T$ with $\dim \Ker T\geq 2.$

For  Question (2), we should mention a remarkable approach to the Invariant Subspace Problem (ISP)   related to the universal operators  in the sense of Rota. The best known example of a universal operator is the adjoint of a shift operator with
infinite multiplicity, such as the adjoint Toeplitz operator $T_{\theta}^*=T_\theta^{-1}: H^2(\mathbb{D})\rightarrow H^2(\mathbb{D})$, with  an infinite Blaschke product $\theta$ (see e.g. \cite{Ca}). In particular, Cowen and Gallardo-Guti\'errez exhibited a class  $T_{\theta}^*$  to suggest an approach to the ISP in \cite{CG}. Furthermore, it is known that $T_{\theta}^*$ is similar to the backward shift $S(1)^*$ on $L^2(0,\infty)$, given by $S(1)^*f(t)=f(t+1)$. Here $S(1)^*$  is a special case of the adjoint semigroup  $\{S(t)^*\}_{t\geq 0}$ defined by  $(S(t)^{*} f)(\zeta)=f(\zeta+t)$. Recall that the shift semigroup $\{S(t)\}_{t\geq 0}$ on $L^2(0,\infty)$ is  \begin{align}(S(t) f)(\zeta)=\left\{
                             \begin{array}{ll}
                               0, & \zeta\leq t, \\
                               f(\zeta-t), & \zeta>t,
                             \end{array}
                           \right.\label{shift}\end{align} which is a $C_0$-semigroup (see, e.g.\cite{Par}). One may refer to \cite{GGPR} for the shift semigroup on weighted spaces $L^2((0,\infty), \tilde{w}(t)dt)$. Now the following commutative diagrams hold for the shift semigroup on $L^2(0,\infty)$ and multiplication semigroups on Hardy spaces.

 \begin{eqnarray*}\begin{CD}
L^2(0,\infty) @>S(t)>> L^2(0,\infty)\\
@VV \ma{L}  V @VV \ma{L} V\\
H^2(\mathbb{C}_+)@>M(t)>> H^2(\mathbb{C}_+)\\
@VV {V^{-1}}    V  @VV  {V^{-1}}   V\\
H^2(\mathbb{D}) @> T(t)>> H^2(\mathbb{D}).
\end{CD}\label{commute1}\end{eqnarray*}
The Laplace transform $\mathcal{L}$ and the isometric isomorphism $V:\; H^2(\mathbb{D})\rightarrow H^2(\mathbb{C}_+)$ are defined as \begin{eqnarray} (\mathcal{L} f)(s)=\int_{0}^\infty e^{-st} f(t)dt,\;\;(V^{-1}g)(z)=\frac{2\sqrt{\pi}}{1+z}g
\left(\frac{1-z}{1+z}\right). \label{V-map}\end{eqnarray}
The multiplication semigroups $\{M(t)\}_{t\geq 0}$ on $H^2(\mathbb{C}_+)$ and $\{T(t)\}_{t\geq0}$ on $H^2(\mathbb{D})$ are defined by \begin{eqnarray}(M(t) g)(s)=e^{-st} g(s)\;\mbox{and}\; (T(t) h)(z)=\phi^t(z) h(z)\label{MT}\end{eqnarray} with  $\phi^{t}(z):=\exp\left(-t(1-z)/(1+z)\right)$. Meanwhile, their adjoint semigroups are
\begin{eqnarray}(M(t)^* g)(s)= P_{H^2(\mathbb{C}_+)}(e^{st} g(s)) \quad \mbox{and}\quad  (T(t)^* h)(z)=P_{H^2(\mathbb{D})}(\phi^{-t}(z) h(z)).\label{M*}\end{eqnarray}

The above fact implies that the answer to Question (2) should have links with the near invariance of the shift semigroup. So we recall the following definition.
\begin{defn}\cite[Definition 1.4]{LP3} Let $\{T(t)\}_{t\geq0}$ be a $C_0$-semigroup in $\ma{B}(\ma{H})$ and $\ma{M}\subseteq\ma{H}$ be a subspace. If for every $f\in \ma{H}$ with $T(t)f \in \ma{M}$ for some $t>0$ we have $f\in \ma{M},$ we call $\ma{M}$ a  nearly $\{T(t)^{*}\}_{t\geq 0}$ invariant subspace.\end{defn}

Note that, by the universality discussion above,  describing the invariant
subspaces of a single $T(t)^*$ may be an intractable problem, and
looking at the common invariant subspaces of a semigroup can be a more fruitful line of
enquiry, as in the work of Lax \cite{lax} on the translation semigroup. The same considerations
apply to nearly invariant subspaces.

In our recent paper \cite{LP3}, we mainly demonstrated a series of prototypical examples for minimal nearly $\{S(t)^*\}_{t\geq0}$ invariant subspaces in $L^2(0,\infty)$, closely related with nearly $\{M(t)^*\}_{t\geq 0}$ invariance on $H^2(\mathbb{C}_+)$. As a subsequent work,   we use the new techniques to determine the nontrivial cyclic nearly invariant subspaces for shift semigroups in much greater generality.

To be specific, the article is organized as follows. In Section 2, we concentrate on the near invariance characterization for a classical discrete semigroup of isometries induced by two or more (even infinitely many) different automorphisms and related questions, especially give the solution to Question (1). Section 3 and Section 4 are committed to  Question (2). We creatively employ the reproducing kernel, minimal Toeplitz kernel and model space to formulate a prototypical class of cyclic nearly invariant subspaces $N(g)=\bigvee\{ge^{-\lambda s}:\;0\leq \lambda \leq \delta\}$
 with $g(s)=1/(1+s)^{n+1}$ $(n\in \mathbb{N}_0)$ in Section 3. These are required for  Section 4,
 which  is concerned with characterizing the subspace $N(g)$ with a rational outer function $g$, which
is expressed in terms of a  model space in
the Hardy space $H^2(\mathbb{C}_+)$. This leads to the more general result for those $g$ where a rational outer function is multiplied by a function invertible in $L^\infty$. The corresponding descriptions in $H^2(\mathbb{D})$ are also addressed.

\section{Near invariance for discrete semigroups and related questions}

In this section, we will use repeatedly the automorphism on $\mathbb{D}$ defined by $$\psi(z)=\lambda \frac{a-z}{1-\overline{a}z},\;|\lambda|=1\;\mbox{and}\;a\in \mathbb{D}.$$

In the subsequent subsections, we first answer  Question (1) and then investigate  some related  issues that arose during this process, including the complementary subspace  $K_{z\theta\circ \psi}\ominus C_\psi(K_\theta)$ and near $S^*$-invariance of $C_\varphi(\Ker T)$ for an inner function $\varphi$ and Toeplitz kernel  $\Ker T$ (including model spaces).

\subsection{Nearly $\{T_{\psi_1^m\psi_2^n}^{-1}:\;m,n\in \mathbb{N}_0\}$ invariant subspaces} In this subsection, we will take two automorphisms $\psi_1, \psi_2$  with zeros $a_1, a_2\in \mathbb{D}$. We are interested in finding the closed subspaces of $ H^2(\mathbb{D})$ that are both nearly invariant with respect to  $T_{\psi_1}^{-1}$ and  $T_{\psi_2}^{-1}$, respectively. In this case $\ma{M}$ is nearly $T_{\psi_1^m\psi_2^n}^{-1}$-invariant for all $m, n\geq 0.$  This forms a finitely-generated commutative isometric operator semigroup denoted by $\{T_{\psi_1^m\psi_2^n}^{-1}:\;m,n\in \mathbb{N}_0\}$. In what follows, we will represent  $\ma{M}$ in terms of a Hitt-like subspace.

For an automorphism $\psi,$ we recall the bounded composition operator $C_\psi:\;H^2(\mathbb{D})\rightarrow H^2(\mathbb{D})$ by $C_\psi f=f\circ\psi.$ Then it holds that   \begin{eqnarray} C_\psi T_z C_\psi^{-1}=C_\psi T_z C_{\psi^{-1}}= T_\psi. \label{unitary1}\end{eqnarray} That is, $T_\psi$ is similar to $T_z$ via $C_\psi$.  By Hitt's result (Theorem \ref{thm Hitt}), $uK_\theta$ is nearly $T_{\bar{z}}$-invariant, where $\theta$ is inner and $K_\theta$ is the model space, $u(0)\neq 0$ and multiplication by $u$ is isometric from $K_\theta$ to $H^2$. So the similarity relationship \eqref{unitary1} yields that $C_\psi (uK_\theta)=(u\circ \psi) C_\psi(K_\theta)$ is nearly $T_{\psi}^{-1}$-invariant   with $(u\circ \psi) (a)=u(0)\neq 0.$  Since the adjoint operator $C_{\psi}^*$ is generally  not a composition operator (see, e.g. \cite[Theorem 9.2]{CM1995}), $C_{\psi}(K_\theta)$ is not a model space in general, even though $C_{\psi}(\theta H^2)$ has the form $(\theta \circ \psi) H^2$. In order to study $C_\psi(K_\theta)$ in detail, we first explore its near $S^*$-invariance. Further note the model space is a special Toeplitz kernel,  so we translate the above question into  research on Toeplitz kernels. The next theorem indicates that the image space of any Toeplitz kernel under a composition operator induced by an automorphism  is also a Toeplitz kernel.

\begin{thm} \label{thm FG} Let $\psi$ be an automorphism and $\Ker T_F$ be a Toeplitz kernel with $F\in L^\infty(\mathbb{T}),$ it follows that
$$C_\psi(\Ker T_F )= \Ker T_G, \;\mbox{where} \;G= (F\circ \psi)\frac{\psi}{z}.$$\end{thm}
\begin{proof}
 We take $f \in \Ker T_F$ and show first that $f\circ \psi \in \Ker T_G$. It follows that $ F f = g$ for some $g \in \bar{z}\overline{H^2}$. So
$$ G(f\circ \psi)=(F\circ \psi)\frac{\psi}{z}(f\circ \psi)=\frac{\psi}{z} (g\circ \psi).$$
Now since $g\in  \overline{z}\overline{H^2}$, then
$$G(f\circ \psi)=\frac{\psi}{z}(g\circ \psi)\in \frac{\psi}{z}(\bar{\psi} \overline{H^2})=\overline{z}\overline{H^2}.$$ This implies $ f\circ \psi \in \Ker T_G.$ So $C_\psi(\Ker T_F ) \subseteq \Ker T_G.$

For the automorphism $\psi^{-1},$ the same argument shows that $C_{\psi^{-1}}(\Ker T_G)\subseteq  \Ker T_H$, where
$$ H=(G\circ \psi^{-1}) \frac{\psi^{-1}}{z}=F\frac{z}{\psi^{-1}}\frac{\psi^{-1}}{z}=F.$$

Since $C_\psi^{-1}=C_{\psi^{-1}},$  we see that $C_\psi$ acts as a bijection between $\Ker T_F$ and $\Ker T_G$, so the desired result follows.
\end{proof}
Letting $F=\overline{\theta}$  in Theorem \ref{thm FG},  a corollary  follows.
\begin{cor} \label{cor model case} Let $\psi$ be  an automorphism and $\theta$ be an inner function, then it follows that
$$C_\psi(K_\theta) =\Ker T_{(\overline{\theta}\circ \psi)\psi/z}.$$ \end{cor}
The following theorem expresses $C_\psi(K_\theta)$   as an invertible function  times a model space.
\begin{thm}\cite[Theorem 6.8]{GMR}\label{unitary} If $\theta$ is an inner function and $\psi$ is an automorphism, then $$f\rightarrow (\sqrt{\psi'}C_\psi)f$$ defines a unitary operator from $K_\theta$ onto $K_{\theta\circ \psi}.$\end{thm}
Summarizing Corollary \ref{cor model case} and Theorem \ref {unitary} yields the following corollary.
\begin{cor}\label{cor equal}Let $\psi$ be an automorphism and $\theta$ be an inner function, then
$$C_\psi(K_\theta) =\Ker  T_{(\overline{\theta}\circ \psi)\psi/z}=\frac{1}{\sqrt{\psi'}}K_{\theta\circ \psi}.$$  \end{cor}

This  indicates  that
\begin{align}C_\psi(K_\theta)=(1-\overline{a}z)K_{\theta\circ \psi}\subsetneq K_{z(\theta\circ \psi)}, \label{CbK}\end{align}  where $\psi(a)=0.$  Next we describe the subspace $C_\psi(K_\theta)$  more concretely.

\begin{rem} $(1)$ If $\theta(0)=0$ then $C_\psi(K_\theta) = K_{z(\theta\circ \psi)/ \psi}$.

This can be deduced from Corollary \ref{cor model case}, since $C_\psi(K_\theta) =\Ker T_{(\overline{\theta}\circ \psi) \psi/z},$ and the complex conjugate of this symbol is $z(\theta\circ \psi)/\psi$, which is an inner function due to $\psi$  divides $\theta\circ \psi$.

$(2)$ If $\theta(0)\neq 0$, it follows $C_\psi(K_\theta)$ is not a model space, which is a proper subspace of   $K_{z(\theta\circ \psi)}.$  Still, using Corollary \ref{cor model case}, we have
$C_\psi(K_\theta)=\Ker T_{(\overline{\theta}\circ \psi)\psi/z} .$ For $$f\in \Ker T_{(\overline{\theta}\circ \psi)\psi/z},\; \mbox{so}\; \frac{(\overline{\theta}\circ \psi)\psi}{z} f=g$$
with some $g\in \bar{z}\overline{H^2},$ then $$\overline{z} (\overline{\theta}\circ\psi)f=g\overline{\psi} \in \bar{z}\overline{H^2},\;\;\mbox{so}\; f\in {\rm Ker }T_{\overline{z}(\overline{\theta}\circ \psi)}=K_{z(\theta\circ \psi)}.$$
\end{rem}
For completeness, we describe the subspace $K_{z(\theta\circ \psi)}\ominus C_\psi(K_\theta)$
explicitly.

 \begin{prop}\label{prop sub} Let $\psi$ be an  automorphism with zero $a\in \mathbb{D}$ and $\theta$ be an inner function, then
\begin{align*}K_{z(\theta\circ \psi)}\ominus C_\psi (K_{\theta})=\mathbb{C}\frac{z\theta\circ \psi -a\theta(0)}{z-a}= \frac{z\theta\circ \psi -a\theta(0)}{\psi}K_\psi. \end{align*}
\end{prop}
\begin{proof}
It is easy to check $$T_{\frac{1}{z-a}}(z(\theta\circ\psi))\in K_{z(\theta\circ \psi)}\ominus (1-\overline{a}z)K_{\theta\circ\psi}=K_{z(\theta\circ \psi)}\ominus C_\psi (K_{\theta})$$ by \eqref{CbK}. Since $T_{\frac{1}{z-a}}(z(\theta\circ\psi))$ is the orthogonal projection of $ z(\theta\circ\psi)/(z-a)$ on $H^2(\mathbb{D}),$  by using the theory of residues the desired result follows.\end{proof}

Here is an example to illustrate the results above.

\begin{exm}  Let  $\theta(z) = b(z) =(a-z)/(1-\overline{a}z)$ with $a\in \mathbb{D}\setminus\{0\}$. That is $\theta(0)=a\neq 0.$  Then $K_b$ is spanned by the reproducing kernel $1/(1-\overline{a}z)$. Now the subspace $C_b(K_b)$ is spanned
by $$\frac{1}{1-\overline{a}b(z)}=\frac{1-\overline{a}z}{1-|a|^2}.$$ This means $C_b(K_b)=(1-\bar{a}z)K_z$ as  shown in \eqref{CbK}. Corollary \ref{cor model case} also confirms that it is the kernel of $T_{(\overline{b}\circ b)b/z}= T_{b/z^2}$, although it is not itself a model space. Moreover, Proposition \ref{prop sub} implies
$$K_{z^2}\ominus C_b(K_b)= \mathbb{C}(z+a).$$
\end{exm}

These results are   summarized in a Hitt-like formula for nearly $T_\psi^{-1}$-invariant subspaces.
\begin{thm}\label{thm Tpsi} Let $\psi$ be an automorphism with zero $a\in \mathbb{D}$, then a space is nearly $T_{\psi}^{-1}$-invariant if and only if it behaves as $C_{\psi}(uK_\theta)=(u\circ \psi)\Ker T_{(\overline{\theta}\circ \psi)\psi/z}$, where $\theta$ is inner, $u(0)\neq 0$ and the multiplication by $u$ is isometric from $K_\theta$ to $H^2.$ Furthermore, it has the following specific form
\begin{align*} \left\{
                                  \begin{array}{ll}
   (u\circ \psi)\Ker T_{(\overline{\theta}\circ \psi)\psi/z}=(u\circ\psi) K_{z(\theta\circ \psi)/\psi}, & \theta(0)=0, \vspace{2mm}\\
        (u\circ \psi)\Ker T_{(\overline{\theta}\circ \psi)\psi/z}=(1-\overline{a}z)(u\circ \psi) K_{ \theta\circ \psi}, & \theta(0)\neq 0.
                  \end{array}
         \right.\end{align*}
  \end{thm}
Theorem \ref{thm Tpsi} further implies the structure of nearly invariant subspaces for discrete semigroups generated by two  different automorphisms on $ \mathbb{D}$.

\begin{thm} \label{discrete}Let $\psi_1$ and $\psi_2$ be two different automorphisms with zeros $a_1\neq a_2\in \mathbb{D}$, then a space is nearly  invariant with respect to the discrete semigroup  $\{T_{\psi_1^m\psi_2^n}^{-1}:\;m,n\in \mathbb{N}_0\}$  if and only if it has the form  $ (u\circ \psi_1)\Ker T_{(\overline{\theta}\circ \psi_1)\psi_1/z} $, where $\theta$ is inner, $u$  satisfies  $u(0)\neq 0$ and $u(\psi_1(a_2))\neq 0$, with multiplication by $u$ is isometric from $K_\theta$ to $H^2.$\end{thm}
\begin{proof}  Theorem \ref{thm Tpsi}  implies the  nearly $T_{\psi_1}^{-1}$-invariant subspace behaves as $$C_{\psi_1}(uK_\theta)=(u\circ \psi_1) \Ker T_{(\overline{\theta}\circ \psi_1)\psi_1/z},\;u(0)\neq 0,$$ and  the multiplication by $u$ is isometric from $K_\theta$ to $H^2.$

Next we show  $C_{\psi_1}(uK_\theta)$ is nearly $T_{\psi_2}^{-1}$-invariant if and only if $u(\psi_1(a_2))\neq 0$.

For sufficiency, suppose a function $f=gh\in C_{\psi_1}(uK_\theta) $ with $g=u\circ \psi_1$ and $h\in \Ker T_{(\overline{\theta}\circ \psi_1)\psi_1/z}$  satisfies $f\overline{\psi_2}\in H^2(\mathbb{D}),$ then we have $f(a_2)=u(\psi_1(a_2)) h(a_2)=0.$ Since $u(\psi_1(a_2))\neq 0$, it follows that $h(a_2)=0$ and so $h\overline{\psi_2}\in \Ker T_{(\overline{\theta}\circ \psi_1)\psi_1/z}$ due to the near invariance of Toeplitz kernels. This yields that
 $$f\overline{\psi_2}=u\circ \psi_1(a_2) h\overline{\psi_2}\in C_{\psi_1}(uK_\theta).$$

For necessity, suppose, on the contrary, that $u(\psi_1(a_2))=0,$ we can always find an $h\in \Ker T_{(\overline{\theta}\circ \psi_1)\psi_1/z}$ with $h(a_2)\neq 0,$ which can be obtained by dividing by a power of $\psi_2.$ Since $C_{\psi_1}(uK_\theta)$ is nearly $T_{\psi_2}^{-1}$-invariant, then $$(u\circ \psi_1)\overline{\psi_2}h=(u\circ \psi_1)\tilde{h}$$ with  $\tilde{h}\in  \Ker T_{(\overline{\theta}\circ \psi_1)\psi_1/z}.$ But that implies $h=\psi_2 \tilde{h}$ and so $h(a_2)=0,$ which is a contradiction.
\end{proof}
 With these at hand, we are in a position to generalize Theorem \ref{discrete}  to the semigroup generated by even infinitely many automorphisms.
\begin{rem} Let $\psi_k$ be  different automorphisms with zeros $a_k\in \mathbb{D}$, $k\in \mathbb{N}$, then a space is nearly  invariant with respect to the discrete semigroup  \[
\{\prod_{i=1}^\infty T_{\psi_i^{m_i}}^{-1}:\;m_i\in \mathbb{N}_0, m_i=0 \hbox{ except for finitely-many } i\}\]
  if and only if it has the form  $ (u\circ \psi_1)\Ker T_{(\overline{\theta}\circ \psi_1)\psi_1/z} ,$ where  $\theta$ is inner, $u$  satisfies  $u(0)\neq 0$ and $u(\psi_1(a_k))\neq 0$, $k\in \mathbb{N}$ with multiplication by $u$ is isometric from $K_\theta$ to $H^2.$ \end{rem}

\subsection{Properties of $C_\varphi(\Ker T)$ for inner functions $\varphi$ and Toeplitz kernels $\Ker T$.}

In this subsection, we take further the study of  whether Theorem \ref{thm FG} holds for  the composition operator induced by a general inner function $\varphi$ that is not an automorphism. Surprisingly, we  have the following theorem.
\begin{thm}\label{thm CphiM}
Let $\mathcal{M}$ be a subspace of $H^2(\mathbb{D})$ of dimension at least $2$, and let
$\varphi$ be an inner function that is not an automorphism on $\mathbb{D}$. Then $C_\varphi(\mathcal{M})$ is not nearly $S^*$-invariant.\end{thm}
Before the proof, we recall  from \cite[Section 2.6]{GMR}, for $\zeta\in \mathbb{D}$, a Frostman shift of an inner function $\varphi$ is defined as
$$\varphi_\zeta(z)=\frac{\zeta-\varphi(z)}{1-\overline{\zeta}\varphi(z)}.$$
It follows that the function $\varphi_\zeta$ is also an inner function for any $\zeta\in \mathbb{D}.$ Frostman's Theorem (see, e.g. \cite[p 45]{Nik}) indicates that $\varphi_\zeta$ is actually a Blaschke product for almost all values of $\zeta \in \mathbb{D}$ (in fact, except for a set
of capacity $0$). Now we can start with a lemma.

\begin{lem} \label{lem alpha} Let $\varphi$ be an inner function that is not an automorphism on $\mathbb{D}$. Then for almost all $\zeta\in \mathbb{D}$ there are distinct points $\alpha,\;\beta\in \mathbb{D}$ such that $\varphi(\alpha)=\varphi(\beta)=\zeta.$
\end{lem}
\begin{proof}
On the one hand, if $\varphi$ is a finite Blaschke product of degree $n\geq 2$ then  $\varphi$ gives an $n$-to-$1$
mapping of $\mathbb{D}$ onto $\mathbb{D}$ (counting multiplicities), which follows
easily from the argument principle as $\varphi(\mathbb{T})$ winds round $\mathbb{T}$ $n$ times. Now if
we exclude the images of the finitely-many points at which $\varphi'=0$  then each
remaining point  $\zeta$ has $n$ distinct preimages under $\varphi.$

On the other hand, if $\varphi$ is an irrational inner function then by Frostman's theorem $\varphi_\zeta(z)=(\zeta-\varphi(z))/(1-\overline{\zeta}\varphi(z))$
is a Blaschke product for almost all $\zeta\in \mathbb{D}$, in which case $\varphi_\zeta(z)=0$ has infinitely many distinct solutions
and so $\varphi(z) =\zeta$ has infinitely many distinct solutions. In sum, there are always  $\alpha\neq \beta\in \mathbb{D}$ such that $\varphi(\alpha)=\varphi(\beta)=\zeta$ for almost all $\zeta \in \mathbb{D}.$ \end{proof}

Now we can proceed with the proof of Theorem \ref{thm CphiM}.
\begin{proof} Since $C_\varphi$ is injective, it follows that  ${\rm dim}\; C_\varphi(\mathcal{M}) \geq 2$, and thus by taking
a nontrivial linear combination of two independent functions in $C_\varphi(\mathcal{M})$, it
contains a function $C_\varphi f$ such that $f \neq 0$ and $(C_\varphi f)(0) = 0.$
Suppose, to get a contradiction, that $C_\varphi(\mathcal{M})$ is nearly $S^*$-invariant. Then
there exists a $g\in  \mathcal{M}$ such that $$(C_\varphi f)(z) = z(C_\varphi g)(z).$$  Now, by Lemma \ref{lem alpha}, for almost all $\zeta\in\mathbb{D}$, there exist $\alpha\neq \beta$ in $\mathbb{D}$ such that $\varphi(\alpha)=\varphi(\beta)=\zeta$. In that case,
$$f(\zeta) = f(\varphi(\alpha)) = \alpha g(\varphi(\alpha)) = \alpha g(\zeta)$$
and similarly $ f(\zeta) = \beta g(\zeta),$ so  $\alpha g(\zeta) = \beta g(\zeta)$. Since $\alpha\neq \beta$ we conclude that
$g(\zeta) = 0$ and so $ f(\zeta) = 0.$ This means   $f(z) = 0$ a.e. for $z\in \mathbb{D}$. Since an $H^2$ function that vanishes a.e. is the zero function (since its Bergman norm is $0$, so its Taylor coefficients are all $0$).  This contradiction proves the theorem. \end{proof}

Since Toeplitz kernels, and in particular model spaces, are always nearly $S^*$-invariant we have the following obvious consequence.

\begin{cor}\label{cor TM} Let $\varphi$ be an inner function that is not an automorphism on $\mathbb{D}$. Then $C_\varphi(\Ker T_F)$  is not a Toeplitz kernel for any $F\in L^\infty(\mathbb{T})$  such that dim $\Ker T_F\geq  2$, and $C_\varphi(K_\theta)$ is not a model space for any model space $K_\theta$ with dim $K_\theta\geq 2.$\end{cor}

Finally, summarizing from  Theorem \ref{thm FG} and Theorem \ref{thm CphiM}, we end this section with near $S^*$-invariance of   $C_\varphi(\Ker T_F)$.
\begin{thm}   Let $F\in L^\infty(\mathbb{T})$  be such that dim $\Ker T_F\geq 2$ and $\varphi$  be an inner function,  then $C_\varphi(\Ker T_F)$ is nearly $S^*$-invariant if and only if $\varphi$ is an automorphism.  \end{thm}

\section{the subspace  $N(1/(1+s)^{n+1})$ in $H^2(\mathbb{C}_+)$}
In this section, we begin to consider  Question (2) by analysing an example which will
be fundamental in describing the general case. To be specific, we creatively employ the reproducing kernel, minimal Toeplitz kernel and model space to reformulate a classical type of cyclic nearly invariant subspace in Hardy spaces over $\mathbb{D}$ and $\mathbb{C_+}.$ Following our recent paper \cite{LP3},  we also denote the smallest (cyclic) nearly $\{S(t)^*\}_{t\geq0}$ invariant subspace containing some nonzero vector $f$ by $[f]_s$. We have proved the initial proposition for $f=e_\delta$ as below.

\begin{prop}\cite[Proposition 2.1]{LP3}\label{prop L2} In $L^2(0,\infty),$ the smallest nearly $\{S(t)^*\}_{t\geq 0}$ invariant subspace containing $e_\delta(\zeta):=e^{-\zeta}\chi_{(\delta,\infty)}(\zeta)$ with some $\delta>0$ has the form $$[e_\delta]_{s}:=\bigvee\{e_\lambda:\;0\leq\lambda\leq \delta\}= L^2(0,\delta) +\mathbb{C}e^{-\zeta}.$$
\end{prop}

We shall use  the following notation:  given $\delta>0$ and $g \in H^2(\mathbb{C}_+)$, let $N(g)$ denote the smallest closed subspace in $H^2(\mathbb{C}_+)$ containing all $g(s)e^{-\lambda s}$ for $0 \leq \lambda \leq \delta$. That is, $$N(g):=\bigvee\{ge^{-\lambda s}:\;0\leq \lambda\leq \delta\}.$$
Similarly, given $h\in H^2(\mathbb{D}),$ the smallest closed subspace in $H^2(\mathbb{D})$ containing all $h(z)\phi^{\lambda }(z)$ for $0 \leq \lambda \leq \delta$ is denoted by $$A(h):=\bigvee\{h\phi^\lambda:\;0\leq \lambda\leq \delta\},$$ where $\phi^\lambda(z)=\exp(-\lambda(1-z)/(1+z))$.

Two corollaries follow on Hardy spaces over $\mathbb{C_+}$ and $\mathbb{D}$.

 \begin{cor}\cite[Corollary 2.3]{LP3}\label{cor half} In $H^2(\mathbb{C}_+),$ the Laplace transform of $[e_\delta]_{s}$ is $$\mathcal{L}([e_\delta]_s)=N\left(\frac{1}{1+s}\right)=K_{\frac{1-s}{1+s}e^{-\delta s}},$$  where $K_{\frac{1-s}{1+s}e^{-\delta s}}$ is a  model space in $H^2(\mathbb{C}_+)$.
\end{cor}
\begin{cor}\cite[Corollary 2.4]{LP3}\label{cor phidelta} In $H^2(\mathbb{D}),$ it holds that $$A(1)=K_{z \phi^{\delta}}.$$\end{cor}

After that, Proposition  \ref{prop L2} was  generalized into more general    $[f_{\delta,n}(\zeta)]_s$ in $L^2(0,\infty)$ with \begin{align*}f_{\delta,n}(\zeta)=\frac{(\zeta-\delta)^n}{n!}
e_\delta(\zeta),\;n\in \mathbb{N}_0.\end{align*}
By the Laplace transform $\ma{L} $ in \eqref{V-map}, it follows that $\ma{L}([f_{\delta,n}(\zeta)]_s)=N(1/(1+s)^{n+1})$. And the corresponding (cyclic) nearly invariant subspaces in Hardy spaces are presented in the following theorem, which is  the generalization of Corollaries   \ref{cor half} and \ref{cor phidelta}.

\begin{thm}\cite[Theorem 3.5]{LP3}\label{thm n} For any $n\in \NN_0$ and $\delta>0$, the following statements are true.

$(1)$ In $H^2(\mathbb{C}_+),$ it holds that the cyclic nearly $\{M(t)^*\}_{t\geq 0}$ invariant subspace $$N\left(\frac{1}{(1+s)^{n+1}}\right)=K_{ \left(\frac{1-s}{1+s}\right)^{n+1}e^{-\delta s}};$$

$(2)$ In $H^2(\mathbb{D}),$ it holds that the cyclic nearly $\{T(t)^*\}_{t\geq 0}$ invariant subspace $$A((1+z)^n)=K_{z^{n+1}\phi^\delta}.$$
\end{thm}

It is observed the above cyclic nearly invariant subspaces all behave as model spaces, which are  invariant subspaces for backward shift semigroup $\{M(t)^*\}_{t\geq 0}$ (defined in \eqref{M*}).  Recall  from \cite[Theorem 3.1.5]{Par},  a non-zero closed subspace $\mathcal{M}\subseteq H^2(\mathbb{C}_+)$ satisfies $M(t) \mathcal{M}\subseteq \mathcal{M}$ for all $t\geq 0$ if and only if $\mathcal{M}=\phi H^2(\mathbb{C}_+)$ for some inner function $\phi\in H^\infty(\mathbb{C}_+).$ So a subspace of  $H^2(\mathbb{C}_+)$ is a model space if and only if it is invariant under  $\{M(t)^*\}_{t\geq 0}$. Next we will recall two lemmas and  provide a new derivation of  the above subspace $N(g)$  precisely. The first lemma  asserts that  every nonzero function in $H^2$  is contained in a minimal Toeplitz kernel.
\begin{lem}\label{minkernel}\cite[Theorem 5.1]{CaP1} Let $f\in H^2\setminus\{0\}$ and let $f=IO$ be its inner-outer factorization. Then there exists a minimal Toeplitz kernel $K$ with $f\in K$, which we write as $K_{
\min}(f)$; moreover, $$K_{
\min}(f)=\ker T_{\overline{z}\overline{IO}/O}.$$\end{lem}
Here the minimal Toeplitz kernel $K_{\min}(f)$ means that $K_{\min}(f)\subseteq \ker T_h$ for any Toeplitz kernel $\ker T_h$ containing $f.$  And then the second lemma is proved in \cite{LP3}, which has many powerful applications.
 \begin{lem}\cite[Lemma 2.7]{LP3} \label{lem gs}If $g(s)$, $sg(s)\in H^2(\mathbb{C}_+)$, then \begin{eqnarray*}sg(s)e^{-st}\in \bigvee\{g(s)e^{-\lambda s},\;|\lambda -t|<\epsilon\}\label{sgs}\end{eqnarray*} for all $\epsilon>0.$  \end{lem}

Recall the reproducing kernel
of $H^2(\mathbb{C}_+)$ is \begin{eqnarray*}
k_w(s)=\frac{1}{2\pi (s+\overline{w})}.\label{kw}
\end{eqnarray*} Then for $w\in \mathbb{C}_+$ and $f\in H^2(\mathbb{C}_+),$ it holds that \begin{eqnarray} f(w)=\langle f, k_w\rangle=\int_{-\infty}^{\infty} f(iy) \overline{k_w(iy)}dy. \label{kernel}
\end{eqnarray}

Now we can reformulate (1) of Theorem \ref{thm n} for  the case $n=0.$

 \begin{prop}\label{model1}
 The subspace $N(1/(1+s))$ is backward shift-invariant and  $$N\Big(\frac{1}{1+s}\Big)=K_{\frac{1-s}{1+s}e^{-\delta s}}.$$
 \end{prop}
 \begin{proof} First we prove $N(1/(1+s))$ is  invariant with respect to $\{M(t)^*\}_{t\geq 0}$.
 Take $0 \le \lambda \le \delta$.
 For the case $ \alpha \leq\lambda,$  $e^{\alpha s} e^{-\lambda s}/(1+s)$ is clearly in $N(1/(1+s))$.  For the other case $\alpha >\lambda$, we use \eqref{kernel} to obtain
 the projection of $e^{\alpha s} e^{-\lambda s}/(1+s)$ into $H^2$ by
 $$
 \left\langle e^{\alpha s} \frac{e^{-\lambda s}}{1+s}, k_{w}\right\rangle=
 \left\langle \frac{1}{2\pi(1+s)}, \frac{e^{(\lambda-\alpha)s}}{s+\overline {w}}\right\rangle=
 \frac{e^{\alpha-\lambda}}{1+w},$$ due to $1/(2\pi(1+s))$ is  a reproducing kernel. In sum, $N(1/(1+s))$ is a model space.
 Further,  Lemma \ref{minkernel} implies that
$$K_{\min}\left(\frac{1}{1+s}e^{-\delta s}\right)=\ker T_{\overline{\frac{1-s}{1+s}e^{-\delta s}}}=K_{\frac{1-s}{1+s}e^{-\delta s}}.$$
Since every Toeplitz kernel containing $e^{-\delta s}/(1+s)$ also contains $e^{-\lambda s}/(1+s)$ for $0 \le \lambda \le \delta$, so the minimal Toeplitz kernel containing $e^{-\delta s}/(1+s)$ equals $N(1/(1+s)),$ ending the proof.\end{proof}
After that,  the derivative-kernel property is used to exhibit (1) of Theorem  \ref{thm n}  for general $n\in \mathbb{N}_0.$
\begin{thm}\label{thm modelway} For any $n\in\NN_0,$ the subspace  $N(1/(1+s)^{n+1})$  is backward shift-invariant and \begin{align}N\Big(\frac{1}{(1+s)^{n+1}}\Big)=K_{ \left(\frac{1-s}{1+s}\right)^{n+1}e^{-\delta s}}.\label{n+1}\end{align} \end{thm}

\begin{proof}By \eqref{kernel}, it holds that
$$\left\langle \frac{1}{2\pi(s+w)}, f\right\rangle = \overline{f(\overline{w})}$$ for any $f \in H^2(\mathbb{C}_+).$  Then differentiating $n$ times, $$\left\langle \frac{(-1)^n n!}{2\pi (s+w)^{n+1}}, f\right\rangle= \overline{f^{(n)}(\overline{w})}.$$
Thus
\begin{eqnarray}
\left\langle f, \frac{1}{2\pi(s+\overline{w})^{n+1}}\right\rangle = (-1)^n\frac{f^{(n)}(w)}{n!}.\label{fms}\end{eqnarray}
Again we take $0 \le \lambda \le \delta$.
On the one hand, if $\alpha\leq \lambda,$ then $ e^{\alpha s} e^{-\lambda s}/ (1+s)^{n+1} \in N(1/(1+s)^{n+1})$.
On the other hand, if $\alpha>\lambda$, using \eqref{fms}, we have
\begin{eqnarray*}
\left\langle \frac{e^{\alpha s} e^{-\lambda s}}{ (1+s)^{n+1}} , k_w \right\rangle &=&
\left\langle \frac{1}{(1+s)^{n+1}} , e^{(\lambda-\alpha)s}k_w \right\rangle \nonumber
\\
&=&\frac{2\pi (-1)^{n}}{n!}\overline{\left(e^{(\lambda -\alpha)s}k_w(s) \right)^{(n)}}\Big|_{s=1},\label{diff}
\end{eqnarray*}
which  has an expression involving a linear combination of the terms $1/(1+w), \cdots$, $1/(1+w)^{n+1}$.

In view of Lemma \ref{lem gs}, we see that $s/(1+s)^{n+1}$ lies in $N(1/(1+s)^{n+1})$ by differentiating with respect to $\lambda$, and hence so does $1/(1+s)^{n}.$
This shows that $$P_{H^2(\CC_+)}\left[e^{\alpha s} \frac{ e^{-\lambda s}}{(1+s)^{n+1}}\right] \in N(1/(1+s)^{n+1})\;\;\mbox{for}\;\;\alpha\geq \lambda.$$  In sum, $N(1/(1+s)^{n+1})$ is backward shift-invariant and so it is a model space.
Similarly, since every Toeplitz kernel containing $e^{-\delta s}/(1+s)^{n+1}$  also contains $e^{-\lambda s}/(1+s)^{n+1}$ for $0 \le \lambda \le \delta$, so the minimal Toeplitz kernel containing $e^{-\delta s}/(1+s)^{n+1}$ equals $N(1/(1+s)^{n+1}).$  Lemma \ref{minkernel} implies the minimal Toeplitz kernel containing $e^{-\delta s}/(1+s)^{n+1} $ is $$K_{\min}\left(\frac{e^{-\delta s}}{(1+s)^{n+1}}\right)=K_{ \left(\frac{1-s}{1+s}\right)^{n+1}e^{-\delta s}},$$ which establishes \eqref{n+1}.
\end{proof}

More generally, given an invertible  function $ G\in H^{\infty}(\mathbb{C}_+)$, although Theorem \ref{thm modelway} implies
$$N\left(\frac{G(s)}{(1+s)^{n+1}}\right)=G(s)N\left(\frac{1}
{(1+s)^{n+1}}\right)=G(s)K_{({\frac{1-s}{1+s}})^{n+1}e^{-\delta s}},$$ we cannot even assert the subspace $N(G(s)/(1+s)^{n+1})$ is backward shift-invariant for the invertible rational function $ G\in H^{\infty}(\mathbb{C}_+)$. In order to formulate $N(G(s)/(1+s)^{n+1})$, we need to recall a result on the analytic multipliers between Toeplitz kernels.

In the following lemma, the analytic function $k\in \mathcal{C}(\ker T_g)$ means $|k^2|dm$ is a Carleson measure for $\ker T_g$, that is, $k \ker T_g \subseteq L^2(\mathbb{T}).$ And we say $\mu$ is a Carleson measure for a subspace $X$ of $H^2(\mathbb{D})$ if there is a constant $C>0$ such that $$\int_{\mathbb{T}}|f|^2 d\mu\leq C\|f\|_2^2\;\;\mbox{for all}\; f\in X,$$ where  $X$ is the Toeplitz kernel, including model space.
 \begin{lem}\label{lem ker}\cite[Theorem 3.3]{CaP3} Let $g, h\in L^\infty(\mathbb{T})$ such that $\ker T_g$ and $\ker T_h$ are nontrivial. For $k\in H(\mathbb{D})$, it holds that $k \ker T_g=\ker T_h$ if and only if $k\in \mathcal{C}(\ker T_g),\;k^{-1}\in \mathcal{C}(\ker T_h)$ and \begin{eqnarray}h=g\frac{\overline{k}}{k}\frac{\overline{q}}
{\overline{p}}\label{hg}\end{eqnarray} for some outer functions $p, q\in H^2(\mathbb{D}).$\end{lem}

Applying Lemma \ref{lem ker} with inner functions $g=\overline{\theta}$ and $h=\overline{\psi}$,  a remark follows.
 \begin{rem}\label{rem model}Suppose that $\theta$ and $\psi$ are inner functions and $k K_{\theta}= K_{\psi}$ with $k\in H(\mathbb{D})$, then it holds that  $K_\psi = \ker T_{\overline{\theta}\overline{k}/k}=K_{\theta k/ \overline{k}}.$\end{rem}
\begin{proof}Firstly, we note $\theta k/\overline{k}$ is unimodular in $L^\infty(\mathbb{T}).$ It also equals $\psi p/q$ from \eqref{hg}, which is in the Smirnov class, so $\theta k/\overline{k}$ lies in $H^\infty(\mathbb{D})$ and inner. This entails $\psi=\gamma \theta k/\overline{k}$, with $\gamma\in \mathbb{C}$ and $|\gamma|=1,$ which is Crofoot's result from \cite{crofoot}.\end{proof}

Now  Lemma \ref{lem ker} is applied to formulate $N(G(s)/(1+s)^{n+1})$  for an invertible rational $G.$
\begin{prop}\label{prop G} Let $ G$ be a rational and invertible function in $H^{\infty}(\mathbb{C}_+)$, then for any non-negative integer $n,$ it holds that \begin{align*}N\left(\frac{G(s)}{(1+s)^{n+1}}\right)=G(s)K_{ \theta}=\ker T_h  \end{align*}   if and only if $G\in \mathcal{C}(\ker T_{\overline{\theta}})$ and $G^{-1}\in \mathcal{C}(\ker T_h)$ with $$\theta(s)= \left(\frac{1-s}{1+s}\right)^{n+1}e^{-\delta s}\;\;\mbox{and}\;\;h=\overline{\theta} \frac{\overline{G }}{G }\frac{\overline{q}}{\overline{p}} $$ for some outer functions $p, q\in H^2(\mathbb{C}_+)$. \end{prop}
\begin{proof} Since $G\in H^\infty(\mathbb{C}_+)$ is rational and invertible, it follows that
$$N\left(\frac{G(s)}{(1+s)^{n+1}}\right)=G(s)K_{\theta}
\;\;\mbox{with}\;\theta(s)= \left(\frac{1-s}{1+s}\right)^{n+1}e^{-\delta s}.$$
Replacing $k$ and $g$ by $G$ and $\overline{\theta}$ in Lemma \ref{lem ker}, the equivalence clearly follows. \end{proof}
Then Proposition \ref{prop G} is used in Example \ref{exm invert}  to show $N(G(s)/(1+s)^{n+1})$ can  sometimes be a Toeplitz kernel but  not a model space for an invertible rational ${G}\in H^\infty(\mathbb{C}_+)$.
\begin{exm}\label{exm invert} It holds that \begin{align}N\left( \frac{s+3}{(1+s)(s+2)} \right)= G K_{ \theta}=\ker T_{\overline{
\theta Gq}/G\overline{p}}\label{M123}\end{align} where   $G(s)=(s+3)/(s+2),\;\theta(s)=(1-s)e^{-\delta s}/(1+s)$ and $p, q$ outer.\end{exm}
\begin{proof}
 Firstly taking out an invertible factor, Proposition \ref{model1} shows that
\begin{eqnarray*}N\left(\frac{s+3}{(1+s)(s+2)}\right)=\frac{s+3}{s+2}
N\left(\frac{1}{1+s}\right)=G K_{\theta} \end{eqnarray*} with $G$ and $\theta$ given in this example.  We conclude that $GK_{\theta}$ is not a model space, otherwise Remark \ref{rem model} implies that $\theta G/ \overline{G}$ would be inner, a contradiction, since it has a pole at $s=3$.

Next we show $GK_{\theta}$ is  a Toeplitz kernel $\Ker T_h$.   Proposition \ref{prop G} yields that
$$h=\bar{\theta}\frac{ \overline{Gq}}{G\overline{p}}$$
with outer functions $p, q$ and $G\in \mathcal{C}(\ker T_{\overline{\theta}})$, $G^{-1}\in \mathcal{C}(\ker T_h)$.  Meanwhile  \begin{eqnarray*} \overline{\theta}\frac{\overline{G}}{G}=e^{\delta s}\frac{1+s}{1-s} \frac{2+s}{2-s}\frac{3-s}{3+s},\end{eqnarray*}
which is clearly a unimodular symbol. So  the subspace in \eqref{M123} follows.
\end{proof}
For completeness, we exhibit the following example with a rational non-invertible $G(s)=1/(s+2)$, which behaves as a model space.
\begin{exm}\label{s+2n} It holds that $$N\left(\frac{1}{(1+s)(s+2)}\right)= K_{{\frac{1-s}{1+s}} \frac{2-s}{2+s} e^{-\delta s}}.$$.\end{exm}
\begin{proof} Using the bijection $V^{-1}$ in \eqref{V-map} to  map $N\left(1/(1+s)(s+2)\right)\subseteq H^2(\mathbb{C}_{+})$ onto $A\left((z+1)/(z+3)\right)\subseteq H^2(\mathbb{D})$, Theorem \ref{thm n} implies that
$$A\left(\frac{z+1}{z+3}\right)=\frac{1}{z+3}K_{z^{2}\phi^{\delta}}.$$

Further employing Remark \ref{rem model}, it holds that
$$\frac{1}{z+3}K_{z^{2}\phi^{\delta}}= K_{z\phi^{\delta}\frac{z+1/3}{1+z/3}}.$$

Then mapping  the above subspace back into $H^{2} (\mathbb{C}_{+})$, it yields the desired model space.
\end{proof}

\section{the  subspace $N(g)$ with a rational outer function $g$}
In this section, we concentrate on finding the subspace $N(g)$ with a rational outer function $g\in
H^2(\mathbb{C}_+)$ noting that $N(\theta h)=\theta N(h)$ for an inner function $\theta.$ The case $g$ is a rational outer function is  fundamental. At this time, $g$ can be factored as the product of an invertible function and one with zeros on the imaginary axis including infinity. By the isometric isomorphism $V^{-1}$ in \eqref{V-map}, we can map the function with zeros in $i\mathbb{R}\cup\{\infty\}$ into a function in $H^2(\mathbb{D})$   with zeros in the unit circle $\mathbb{T}$. In what follows, we turn to formulate $A(h)$ for  a rational function $h$ with $n$ zeros on   $\mathbb{T}$.  In our recent paper \cite{LP3}, we obtained the following proposition.
\begin{prop}\cite[Proposition 3.10]{LP3} \label{prop poly}
Let $\widetilde{p}_n(z):=\prod_{j=1}^n (z+w_j)$ with $w_j\in \mathbb{T}$,   $j=1,\cdots, n$;
then it follows that \begin{equation} A(\widetilde{p}_n)+ \phi^\delta K_{z^n}
=K_{z^{n+1}\phi^\delta}.\label{pN}\end{equation}  \end{prop}

In order to explore the concrete description of $A(\widetilde{p}_n)$,  several lemmas are cited for properties of model spaces.

\begin{lem}\cite[Proposition 5.5]{GMR}\;\label{lem Hinfinity}If $\varphi\in H^\infty$ and $\theta$ is inner, then $T_{\overline{\varphi}}K_\theta\subseteq K_\theta.$ \end{lem}
\begin{lem}\cite[Lemma 5.10]{GMR}\;\label{u12} Let $\theta_1$ and $\theta_2$ be inner. Then $$K_{\theta_1\theta_2}=K_{\theta_1}\oplus \theta_1K_{\theta_2}.$$ \end{lem}
\begin{lem}\cite[Proposition 5.4]{GMR}\;\label{lem fgu} For an inner function $\theta$, the model space $K_\theta$ is the set of all $f\in H^2$ such that $$f=\overline{gz}\theta$$  almost everywhere on $\mathbb{T}$ for some $g\in H^2.$ \end{lem}
\begin{rem} \label{rem fg}From the above formula, it also holds $g=\overline{fz}\theta,$ so $g\in K_{\theta}$ too. \end{rem}

The  proposition below is  crucial, which involves Proposition  \ref{prop poly} with $n=1$.
\begin{prop}\label{prop z+w}For $w\in \mathbb{T},$ it holds that $$A(z+w)=K_{z^2\phi^\delta}.$$ \end{prop}

 \begin{proof} We will compute the orthogonal complement   of $A(z+w)$ in $K_{z^2\phi^\delta}.$ Lemma \ref{u12} and Proposition \ref{prop poly} imply
 $$A(z+w)+\phi^\delta K_z=K_{z^2\phi^\delta}=K_{z\phi^\delta}\oplus z\phi^\delta K_z.$$

On the one hand, for the vector  $\alpha z\phi^\delta\in z\phi^\delta K_z,\;\alpha\in \mathbb{C}$, it holds that
$$\langle (z+w)\phi^\lambda, \alpha z\phi^\delta\rangle=\langle 1, \alpha\phi^{\delta-\lambda}\rangle-w\langle 1, \alpha z\phi^{\delta-\lambda}\rangle
=\overline{\alpha\phi^{\delta-\lambda}(0)}=0$$ for all $0\leq \lambda \leq \delta$ if and only if $\alpha=0$. This means $$z\phi^\delta K_z \cap [A(z+w)]^{\perp}=\{0\}.$$

On the other hand, we assume $f\in K_{z\phi^\delta}$ such that $\langle (z+w)\phi^\lambda, f\rangle=0$ for all $0\leq \lambda \leq \delta.$ By Lemma \ref{lem fgu}, we suppose $f=\overline{gz} z\phi^\delta=\overline{g}\phi^\delta$ with $g\in K_{z\phi^\delta}$ by Remark \ref{rem fg}.  It turns out
\begin{align*} \langle (z+w)\phi^\lambda, f\rangle&= \langle (z+w)\phi^\lambda, \overline{g}\phi^\delta \rangle\\&=\langle (z+w)g,  \phi^{\delta-\lambda} \rangle=0,\end{align*} for all $0\leq \lambda \leq \delta.$  Since $\bigvee\{\phi^{\delta-\lambda}:\; 0\leq \lambda \leq \delta\}=K_{z\phi^\delta}$ (see \cite[Corollary 2.4]{LP3}), we have that
$(z+w)g \in z\phi^\delta H^2$. By the uniqueness of inner-outer factorization, we suppose $$(z+w)g(z)=z\phi^\delta h$$ for some $h\in H^2.$ Considering $g\in K_{z\phi^\delta}=\ker T_{\overline{z\phi^\delta}},$ we obtain   $$\overline{z\phi^\delta}(z+w)g=h\in (z+w)\overline{z H^2}.$$ We further denote $h =(z+w)\overline{z k}=(1+w\overline{z})\overline{k}\in H^2$ for some $k\in H^2.$ This gives $k=0$, $h=0$ and $g=0$, then $f=0$. Thus $$ K_{z \phi^\delta}\cap [A(z+w)]^{\perp}=\{0\}. $$

Finally, we take $f+\alpha z\phi^\delta\in K_{z^2\phi^\delta}$ with $f=\overline{g}\phi^\delta\in K_{z\phi}$ and $\alpha\neq 0$ such that
\begin{align*}&\langle (z+w)\phi^\lambda, f+\alpha z\phi^\delta\rangle\\&=\langle \overline{\alpha}, \phi^{\delta-\lambda} \rangle+\langle (z+w)g,\phi^{\delta-\lambda}\rangle\\&=\langle \overline{\alpha}+(z+w)g, \phi^{\delta-\lambda}\rangle=0\;\;\mbox{for all}\;0\leq \lambda\leq \delta.\end{align*}
This implies $\overline{\alpha}+(z+w)g\in z\phi^\delta H^2.$ Still assume $\overline{\alpha}+(z+w)g= z\phi^\delta \hat{h}$ for some $ \hat{h}\in H^2.$ And then  $(z+w)g= z\phi^\delta  \hat{h}-\overline{\alpha}$. Due to $g\in K_{z\phi^\delta},$ we have that \begin{align*} (z+w)\overline{z\phi^\delta}g= \overline{z\phi^\delta}(z\phi^\delta  \hat{h}-\overline{\alpha})\in (z+w) \overline{zH^2}. \end{align*} This, together with $w\in \mathbb{T}$, imply that  $$ \hat{h}-\overline{\alpha z\phi^\delta}
\in  \overline{(z+ \overline{w}^{-1})H^2}\Rightarrow  \hat{h}-\overline{\alpha z\phi^\delta}
\in  \overline{H^2}.$$ Since $\alpha\neq 0,$ it follows that $ \hat{h}\in H^2\cap \overline{H^2}.$ So $ \hat{h}= \hat{h}(0):=\beta\in \mathbb{C}.$ So that $(z+w)g=\beta z\phi^\delta -\overline{\alpha}$ and then $$\overline{(z+w)}f=\overline{(z+w)g}\phi^\delta
=(\overline{\beta z\phi^\delta} - \alpha)\phi^\delta =\overline{\beta z}-\alpha\phi^\delta.$$
Since $f\in H^2,$ so letting $z\rightarrow-w$ in the above formula, it follows that $\overline{\beta } =-w \alpha\phi^\delta(-w).$ And then it holds that $$(\overline{z+w})f=\alpha( -w \phi^\delta(-w)\overline{z}-\phi^\delta),\;\alpha\in \mathbb{C}\setminus\{0\}.$$ Letting $z=0$ in the above formula, it follows that
\begin{align} f(0)=-\alpha w e^{-\delta}. \label{f0}\end{align}
Next we show $f\notin K_{z\phi^\delta}.$ On the contrary, suppose $f\in K_{z\phi^\delta},$ Lemma \ref{lem Hinfinity} implies $\overline{z}(f-f(0))+ \overline{w}f=(\overline{z+w})f-\overline{z}f(0)\in K_{z\phi^\delta}$. It turns out that $$\alpha( -w \phi^\delta(-w)\overline{z}-\phi^\delta)-f(0)\overline{z} \in K_{z\phi^\delta},$$ implying  $f(0)=-\alpha w\phi^\delta (-w)$ due to $\phi^\delta \in K_{z\phi^\delta}$. This  contradicts  with \eqref{f0} due to $w\in \mathbb{T}$, so $f\notin K_{z\phi^\delta}.$ In sum, we conclude $A(z+w)=K_{z^2\phi^\delta}.$
\end{proof}

Now the concrete form of $A\Big(\prod_{j=1}^n (z+w_j)\Big)$ with $w_j\in \mathbb{T}$,   $j=1,\cdots, n$ can be determined by Proposition \ref{prop z+w} and mathematical induction. This improves \eqref{pN}.

\begin{prop} \label{prop pn}Let $\widetilde{p}_n(z):=\prod_{j=1}^n (z+w_j)$ with $w_j\in \mathbb{T}$,   $j=1,\cdots, n,$ it follows that \begin{equation}
A(\widetilde{p}_n)= K_{z^{n+1}\phi^\delta}.\label{pn}\end{equation}
\end{prop}
\begin{proof}
 Proposition \ref{prop z+w} implies \eqref{pn} holds for $n=1$. And then we suppose \eqref{pn} holds for the case $n-1$, we will prove for the general  $n$.
Firstly, we have
\begin{align*} A(\widetilde{p}_n)= \overline{(z+w_n)A(\widetilde{p}_{n-1}) }=\overline{(z+w_n) K_{z^n\phi^\delta}}\subseteq \overline{zK_{z^n \phi^\delta}+K_{z^n\phi^\delta}}.\end{align*}
Since $$\langle \prod_{j=1}^{n-2}(z+w_j)\phi^\lambda,z^n\phi^\delta \rangle=0\;\;\mbox{and}\;\;\langle z \prod_{j=1}^{n-2}(z+w_j)\phi^\lambda,z^n\phi^\delta \rangle=0 $$  for $0\leq \lambda \leq \delta,$ it follows that
 $$z(z+w_n)\prod_{j=1}^{n-2}(z+w_j)\phi^\lambda,\; (z+w_n)\prod_{j=1}^{n-2}(z+w_j)\phi^\lambda\in \overline{(z+w_n) K_{z^n\phi^\delta}}$$ for $0\leq \lambda \leq \delta.$ The result of the case $n-1$ implies that
$$A\Big((z+w_n)\prod_{j=1}^{n-2}(z+w_j)\Big)=K_{z^n\phi^\delta}.$$  Thus it follows that
$$ \overline{zK_{z^n \phi^\delta}+K_{z^n\phi^\delta}}\subseteq A(\widetilde{p}_n).$$  In sum, we conclude that\begin{align} A(\widetilde{p}_n)=\overline{zK_{z^n \phi^\delta}+K_{z^n\phi^\delta}}.\label{sum}\end{align}

Now for any $f\perp A(\widetilde{p}_n),$ \eqref{sum} implies $f\perp K_{z^n\phi^\delta}$ and $f\perp zK_{z^n\phi^\delta},$ which means $f\in z^n\phi^\delta H^2(\mathbb{D})$ and $S^*f\in z^n\phi^\delta H^2(\mathbb{D}).$ So we can suppose $f=z^n\phi^\delta h$ with some $h\in H^2(\mathbb{D}),$ and then $S^*f=z^{n-1}\phi^\delta h\in z^n\phi^\delta H^2(\mathbb{D}),$ verifying $z$ divides $h$. Hence $f\in z^{n+1}\phi^\delta H^2(\mathbb{D}),$ this shows $$K_{z^{n+1}\phi^\delta}\subseteq A(\widetilde{p}_n).$$ Further,  since $K_{z^n\phi^\delta}\subseteq K_{z^{n+1}\phi^\delta}$ and $zK_{z^n\phi^\delta}\subseteq K_{z^{n+1}\phi^\delta},$ so combining with \eqref{sum} we obtain $A(\widetilde{p}_n)\subseteq K_{z^{n+1}\phi^\delta}$. In sum, we verify \eqref{pn}.\end{proof}
 For a general  rational outer function $h\in H^2(\mathbb{D})$, it can be factored as the product of an invertible function and a function with zeros on the unit circle. By this time, combining Remark \ref{rem model}, we can present  the characterization  of $A(h)$ which essentially improves  \cite[Theorem 3.11]{LP3}.

\begin{thm}\label{thm improvedgeneralg} Suppose $h(z)=\widetilde{p}_n(z)q(z)$ with $\widetilde{p}_n(z):=\prod_{j=1}^n (z+w_j)$, where $w_j\in \mathbb{T}$,  $j=1,\cdots, n,$ and $q$ is an invertible rational function in $L^\infty(\mathbb{T})$. Then
\begin{align*}A(h) =qK_{z^{n+1}\phi^\delta}. \end{align*} \end{thm}

Accordingly, a general rational function $g\in H^2(\mathbb{C}_+)$ can be decomposed as $g=G_1G_2$, where $G_1$ is invertible in $L^\infty(i \mathbb{R})$ and $G_2$ only has zeros in $i\mathbb{R}\cup \{\infty\}.$ We can further make the denominator of $G_2$ equal to a power of $(1+s)$ and there will always be at least $1$ as the function is in $H^2(\mathbb{C}_+).$ Now suppose the degrees of the numerator and denominator of $G_2$ are $m$ and $n$, respectively. This means $m$ is the number of imaginary axis zeros of $g$ and $g$ is asymptotic to $s^{m-n}$ at $\infty$ so $n>m.$ Particularly, we write $G_2(s)=\prod_{k=1}^m(s-y_k)/(1+s)^n$ with all $y_k\in i\mathbb{R}$. Using the relation between $H^2(\mathbb{D})$ and $H^2(\mathbb{C}_+)$, we deduce the formula for  $$N(g)=N(G_1G_2)=G_1N\left(
\frac{\prod_{k=1}^m(s-y_k)}{(1+s)^n}\right)\;\;\mbox{with}\;n>m.$$
Now for integers $n>m\geq 1$,   Proposition \ref{prop pn} implies the subspace $
 N\left( \prod_{k=1}^m(s-y_k)/(1+s)^{n}\right)$  corresponds to
$$A\Big( (z+1)^{n-m-1}\prod_{k=1}^m(z-w_k) \Big)=K_{z^{n}\phi^\delta}$$ with $w_k=(1-y_k)/(1+y_k)\in \mathbb{T},\;k=1, \cdots, m.$ So mapping by the isometric isomorphism $V$, the proposition below is true.

\begin{prop} For integers $n>m\geq1,$ and $y_k\in i\mathbb{R},$ $k=1,2,\cdots,m$, it holds that \begin{align*} N\left(\frac{\prod_{k=1}^m(s-y_k)}{(1+s)^{n}}\right)
=K_{\left(\frac{1-s}{1+s}
\right)^{n} e^{-\delta s}}.\label{mn}\end{align*}  \end{prop}

In sum, we can obtain the characterization for $N(g)$  with a rational outer function $g\in H^2(\mathbb{C}_+),$  which essentially improves  \cite[Theorem 3.13]{LP3}.
\begin{thm}\label{thm generalC+G}
Let $g\in H^2(\mathbb{C}_+)$ be rational with $m$ zeros on the imaginary axis and let $n>m$ such that $s^{n-m} g(s)$ tends to a finite nonzero limit at $\infty.$  Then $g$ can be written as $g=G_1G_2$, where $G_1$ is rational and invertible in $L^\infty(i\mathbb{R})$ and $G_2(s)=\prod_{k=1}^m(s-y_k)/(1+s)^n$ with all $y_k\in i\mathbb{R}$.  Then it holds that
\begin{align*}N(g)=G_1(s)K_{(\frac{1-s}{1+s})^{n}e^{-\delta s}}.\end{align*}\end{thm}

For a general function $h\in H^2(\mathbb{D})$, there follow two remarks on $A(h)$.
\begin{rem}(1) Suppose  $h\in H^2(\mathbb{D})$ has an inner factor, then  $A(h)$ is still of the form $uK$ with model space $K$ but now $u$ is not outer. So $A(h)$ is not a Toeplitz kernel due to the near invariance of Toeplitz kernels.

(2) Generally, factorize the function $h$ as $h=\theta u p$ where $\theta$ is inner and $u$ is invertible in $H^\infty$ (hence outer). Then $p$ is non-invertible and outer. If $p$ is rational then it has factors $z+w_j$ with $w_j\in \mathbb{T}$ and the characterization of $A(h)=\theta u  A(p)$ follows from Theorem \ref{thm improvedgeneralg}. \end{rem}

As the  extension of Proposition \ref{prop L2},   Theorem \ref{thm generalC+G} is applied to concretely illustrate two smallest nearly $\{S(t)^*\}_{t\geq 0}$ invariant subspaces containing the functions related with $e_\delta(\zeta):=e^{-\zeta}\chi_{(\delta,\infty)}(\zeta)$.
\begin{exm} (1) For $m \in \mathbb{N}$ and $\delta>0,$  consider the smallest cyclic subspace $[f_m]_s$ containing the function \begin{align*}f_m(\zeta)=\sum_{k=0}^{m}f_{\delta ,k}(\zeta)=\sum_{k=0}^{m}\frac{(\zeta-\delta)^{k}}{k!}e_{\delta}(\zeta)\in L^2(0,\infty).
\end{align*}
By the Laplace transform $\mathcal{L}$ in \eqref{V-map}, we map $[f_m]_s$ onto
\begin{align*}
 N\left(\sum_{k=0}^m \frac{1}{(1+s)^{k+1}}\right)= F_{m}(s) K_{\frac{1-s}{1+s} e^{-\delta s}},
\end{align*} where  $$F_{m}(s):= \sum_{k=0}^m \frac{1}{(1+s)^{k+1}}=  \frac{(1+s)^{m+1}-1}{s(1+s)^{m+1}}$$ is rational invertible in $H^2(\mathbb{C}_+)$.\vspace{1mm}

(2) For $m \in \mathbb{N}$ and $0<\delta_{1}<\delta_{2}<\cdots<\delta_{m},$   consider the smallest cyclic subspace $[g_{m}]_s$ containing the  function  $$g_{m}(\zeta)=\sum_{k=1}^{m} e_{\delta_k}(\zeta)\in L^2(0,\infty).$$
Similarly, employing the Laplace transform $\mathcal{L}$ in \eqref{V-map},  $[g_{m}]_s$ is mapped onto
\begin{align*} N\left(\sum_{k=1}^{m} e^{-\left(\delta_{k}-\delta_{1}\right)} e^{-\left(\delta_{k}-\delta_{1}\right) s}\right) = H_{m}(s) K_{\frac{1-s}{1+s} e^{-\delta_{1} s}},\end{align*}
 where \begin{align*}H_{m}(s):=\sum_{k=1}^{m} e^{-\left(\delta_{k}-\delta_{1}\right)} e^{-\left(\delta_{k}-\delta_{1}\right) s} \end{align*} is an invertible function. \end{exm}

\emph{Open Question.} How can one characterize the subspace $A(q)\subseteq H^2(\mathbb{D})$ when $q$ is non-invertible with an irrational outer factor?

\end{document}